\documentclass[english]{amsart}

\usepackage{esint}
\usepackage[svgnames]{xcolor}
\usepackage[colorlinks,citecolor=red,pagebackref,hypertexnames=false,breaklinks]{hyperref}
\usepackage{pgf,tikz}
\usepackage{pdfsync}

\usepackage{dsfont}
\usepackage{url}
\usepackage[utf8]{inputenc}
\usepackage[T1]{fontenc}
\usepackage{lmodern}
\usepackage{babel}
\usepackage{mathtools}
\usepackage{amssymb}
\usepackage{lipsum}
\usepackage{mathrsfs}
\usepackage{color}

\newtheorem{theorem}{Theorem}[section]
\newtheorem{proposition}{Proposition}[section]
\newtheorem{lemma}{Lemma}[section]

\newtheorem{corollary}{Corollary}[section]

\newtheorem{remark}{Remark}[section]

\numberwithin{equation}{section}

\title[hybrid inverse problems]{Some stability inequalities for hybrid inverse problems}

\author[Mourad Choulli]{Mourad Choulli}
\address{Universit\'e de Lorraine}
\email{mourad.choulli@univ-lorraine.fr}

\thanks{The author is supported by the grant ANR-17-CE40-0029 of the French National Research Agency ANR (project MultiOnde). }

\date{}

\begin{document}

\begin{abstract}
We study some hybrid inverse problems associated to BVP's for Schr\"odinger and Helmholtz type equations. The inverse problems we consider consist in the determination of coefficients from the knowledge of internal energy densities. We establish local Lipschitz stability inequalities as well as H\"older stability inequalities.
\end{abstract}



\maketitle

\section{Introduction}\label{section1}

Coupled-physics or hybrid inverse problems have attracted many researchers from the inverse problem community during the last decades. 

Among results concerning these hybrid inverse problems we quote those established by Bal and Uhlmann \cite{BU} in the isotropic case. Precisely, they considered the quantitative photoacoustic tomography. They showed that there exists an open subset of illuminations for which we have Lipschitz stability for the determination of two medium parameters  from $2n$ boundary measurements in a $n$-dimensional space. They also proved that two measurements are sufficient provided that the domain under consideration satisfies an extra geometric condition. Using a different method, Alessandrini, Di Cristo, Francini and Vessella \cite{ADFV} proved a H\"older stability estimate in the case of two well chosen illuminations. Recently, Bonnetier, Choulli and Triki \cite{BCT}
established also a H\"older stability estimate in the case of two arbitrary pointwise sources generating two illuminations. This situation is more suitable for real physical problems. The determination of the absorption coefficient from a single measurement was already studied by Choulli and Triki \cite{CT1,CT2}. For this problem we got a weighted H\"older stability estimate.

We just quoted few results for the quantitative photoacoustic tomography.  We refer to Alberti and Capdeboscq \cite{AC} and references therein for a complete overview  concerning recent progress dealing with hybrid  inverse problems for both isotropic and anisotropic cases.

We discuss in the present work the stability issue for various kind of hybrid inverse problems that lead to the same BVP. For possible applications we list below four examples.

\subsection{Quantitative  photoacoustic tomography} It is a hybrid imaging modality where high frequency electromagnetic waves are combined with ultrasounds. Precisely, the medium is illuminated by high frequency electromagnetic wave (e.g. laser). A part of the electromagnetic radiation is then absorbed by the tissues and therefore dissipated partially into heat. The increase of temperature produces an expansion of the medium and in consequence creates acoustic waves. From a mathematical point of view, the first step consists in determining the absorbed electromagnetic energy density $H$ from boundary measurements. It turns out that $H$ is the initial condition of an acoustic wave equation. This is typically a control problem which is already solved and many results can be found in the literature devoted to control theory.  Once we recover $H$, the objective is then the determination of the diffusion coefficient or the diffusion matrix $\mathfrak{a}$ and the absorption  coefficient $\mathfrak{q}$ from the energy
\[
H(x)=G(x)\mathfrak{q}(x)u(x),\quad x\in \Omega .
\]
where $\Omega$ is the domain occupied by the medium,  $G$ is the Gr\"uneisen parameter  and $u$ is the light intensity. In the diffusive regime, it is shown that $u$ is the solution of the BVP 
\begin{equation}\label{e1}
\left\{
\begin{array}{l}
-\mbox{div}(\mathfrak{a}\nabla u)+\mathfrak{q}u=0\quad \mbox{in}\; \Omega,
\\
u_{|\Gamma}=f.
\end{array}
\right.
\end{equation}
Here $\Gamma$ is the boundary of $\Omega$, $\mathfrak{a}$ is the diffusion coefficient in the isotropic case or the diffusion matrix in the anisotropic case and $f$ represents the illumination.

We assume in the present work that the parameter $G$ is known. We set for simplicity $G=1$. 

In fact, several energy densities $H$ may be necessary to recover $\mathfrak{a}$ and $\mathfrak{q}$. Each energy density corresponds to a different illumination.

\subsection{Quantitative dynamic elastography}

In the quantitative dynamic elastography we want to recover the tissue parameters from the tissue displacement.
 In the simplified  scalar elastic model $u$, one component of the displacement, is the solution of the BVP \eqref{e1} in which $\mathfrak{a}$ is the shear modulus and $\mathfrak{q}=-\rho k^2$, where $\rho$ is the density and $k$ is the frequency. We assume that $k>0$ is known. Then the objective in that inverse  problem is the determination of the shear modulus and the density from the knowledge of values of $u$ corresponding to different values of the boundary data $f$.

\subsection{Microwave imaging by elastic deformation}

The construction of the conductivity in the context of electrical impedance tomography from boundary measurements is well known to be severely ill-conditioned. 
It is shown in Ammari, Capdeboscq, de Gournay, Triki and Rozanova-Pierrat \cite{ACGTR} that the boundary measurements with simultaneous localized ultrasonic perturbations allow  the recovery of the conductivity with good resolution. In that case the  electrical impedance tomography is substituted by the problem of reconstructing the conductivity (and permittivity)  from internal electrical energy densities.
In the microwave regime, this problem leads again to the BVP \eqref{e1}, where $\mathfrak{a}$ is the conductivity and $\mathfrak{q}=-k^2\mathfrak{p}$, $k>0$ is  the frequency and $\mathfrak{p}$ is the permittivity. The internal data is given by various electrical energy densities of the form $H=\mathfrak{p}u^2$ or $H=\mathfrak{a}|\nabla u|^2$, corresponding to several boundary data $f$.

In the preceding two examples, we assume that the frequency $k$ is known. We suppose for simplicity that $k=1$.

\subsection{Acousto-optic imaging}

When a medium is exited by an acoustic radiation then its optical properties are modified. It is known that in this case the scattered field carries informations about the medium. This principle was the basis for the development of a hybrid imaging modality, known as acousto-optic imaging. In the simplified model, the electromagnetic energy density $u$ solves the BVP \eqref{e1} in which $\mathfrak{a}$ is the diffusion coefficient and $\mathfrak{q}$ is the absorption coefficient (e.g. Bal and  Schotland \cite{BS}). The measured internal energy density for the actual inverse problem is given by $H=\mathfrak{q}u^2$.

\subsection{Main notations and definitions}

In the rest of this text we use the following notations. The norm of a Banach space $E$ is denoted by $\|\cdot \|_E$. The ball, of a Banach space $E$, of center $x_0\in E$ and radius $r>0$ will denoted by $B_E(x_0,r)$. If $E$ and $F$ are two Banach spaces, the norm of $\mathscr{B}(E,F)$, the Banach space of linear bounded operators between $E$ and $F$, will denoted by $\|\cdot \|_{\rm op}$.

Unless otherwise specified $\Omega$ is a $C^{1,1}$ bounded domain of $\mathbb{R}^n$ ($n\geq 2$) with boundary $\Gamma$. We denote by $\lambda_1(\Omega)$  the first eigenvalue of the Laplace operator on $\Omega$ with Dirichlet boundary condition.

For $\mu \ge 1$ and $\lambda \ge 0 $, define
\begin{align*}
&\mathcal{A}_\mu=\left\{\mathfrak{a}=(a^{k\ell})\in C^{0,1}\left(\overline{\Omega},\mathbb{R}^{n\times n}\right);\; \mathfrak{a}\; \mbox{is symmetric and} \right.
\\
&\hskip 6cm \left. \mu ^{-1}|\xi|^2\le (\mathfrak{a}\xi|\xi)\le \mu |\xi|^2\; \mbox{for all}\; \xi \in \mathbb{R}^n\right\},
\\
&\mathcal{A}_\mu^s=\{\mathfrak{a}\in C^{0,1}(\overline{\Omega});\; \mathfrak{a}I_n\in \mathcal{A}_\mu\},
\\
&\mathcal{Q}_\lambda =\{\mathfrak{q}\in L^\infty(\Omega ) ;\; \mathfrak{q}\ge -\lambda \},
\\
&\mathcal{Q}_\lambda ^-=\{\mathfrak{q}\in L^\infty(\Omega ) ;\; 0\ge \mathfrak{q}\ge -\lambda \}.
\end{align*}
Here $(\cdot |\cdot )$ is the Euclidian scalar product on $\mathbb{R}^n$ and $I_n$ is the  $n\times n$ identity matrix.

Let $\mathcal{J}=\{(\mu ,\lambda )\in [1,\infty )\times [0,\infty );\; \mu \lambda <\lambda_1(\Omega )\}$. Define, for all $(\mu ,\lambda)\in \mathcal{J}$, $\mathcal{D}_{\mu ,\lambda}=\mathcal{A}_\mu\times \mathcal{Q}_\lambda$, $\mathcal{D}_{\mu ,\lambda}^\bullet=\mathcal{A}_\mu^s\times \mathcal{Q}_\lambda^-$ and
\[
\mathcal{D}= \bigcup_{(\mu ,\lambda)\in \mathcal{J}}\mathcal{D}_{\mu ,\lambda} .
\]

\subsection{Local Lipschitz stability}

Let $r\ge 2$. We prove in Corollary \ref{corollary1.1} (Subsection \ref{subsection2.1}) that, for any $(\mathfrak{a},\mathfrak{q})\in \mathcal{D}$ and $f\in W^{2-1/r,r}(\Gamma)$, the BVP \eqref{e1} has a unique solution $u_{\mathfrak{a},\mathfrak{q}}\in W^{2,r}(\Omega)$.

\begin{theorem}\label{main-theorem1}
Suppose that, for some $0<\theta <1$, $\Omega $ is of class $C^{2,\theta}$, $\mathfrak{a}_0\in C^{1,\theta}(\overline{\Omega},\mathbb{R}^{n\times n})\cap \mathcal{A}_\mu$, for some $\mu \ge 1$, and $f\in C^{2,\theta}(\Gamma)$ satisfies $f>0$. Let $j=1$ or $j=2$. Then there exists $\mathcal{U}$, a neighborhood of $0$ in $L^\infty (\Omega)$, and a constant $C=C(n,\Omega, \mu  ,f,r)$ so that, for any $\mathfrak{q}$, $\tilde{\mathfrak{q}}\in \mathcal{U}$, we have
\begin{equation}\label{in1}
\|\mathfrak{q}-\tilde{\mathfrak{q}}\|_{L^\infty (\Omega )}\le C \|\mathfrak{q}u_{\mathfrak{q}}^j-\tilde{\mathfrak{q}}u_{\tilde{\mathfrak{q}}}^j\|_{L^\infty (\Omega )},
\end{equation}
where $u_{\mathfrak{q}}=u_{\mathfrak{a}_0,\mathfrak{q}}$ and $u_{\tilde{\mathfrak{q}}}=u_{\mathfrak{a}_0,\tilde{\mathfrak{q}}}$.
\end{theorem}

In this theorem the case $\mathfrak{q}\ge 0$, $\tilde{\mathfrak{q}}\ge 0$ and $j=1$ corresponds to a stability inequality for the problem that consists in the determination of the absorption coefficient $\mathfrak{q}$ in the quantitative photoacoustic tomography, from the energy density $H=\mathfrak{q}u_{\mathfrak{q}}$. The case $\mathfrak{q}\ge 0$, $\tilde{\mathfrak{q}}\ge 0$ and $j=2$ gives a stability inequality for the problem of determining the absorption coefficient in the acousto-optic imaging, from the energy density $H=\mathfrak{q}u_{\mathfrak{q}}^2$. While the case $\mathfrak{q}\le 0$, $\tilde{\mathfrak{q}}\le 0$ and $j=2$ contains a stability result for the determination of the permittivity in the microwave imaging, from the energy density $H=\mathfrak{q}u_{\mathfrak{q}}^2$.

Next, we state a local Lipschitz stability estimate that applies to the problem of determining the absorption coefficient in the acousto-optic imaging from the knowledge  of an internal data.  
Prior to doing that, we introduce a definition. Fix $\mu \ge 1$ and $0< \underline{\mathfrak{q}}< \overline{\mathfrak{q}}<\mu^{-1} \lambda_1(\Omega)$,  and define
\[
\mathbf{Q}=\{ \mathfrak{q}\in  L^\infty(\overline{\Omega});\; \underline{\mathfrak{q}}\le \mathfrak{q}\le \overline{\mathfrak{q}}\}.
\]

\begin{theorem}\label{main-theorem3}
Assume that $r=2$ and $f$ satisfies $\mathrm{essinf}\, f=m>0$ on $\Gamma$. Let $\mathfrak{a}\in \mathcal{A}_\mu$. For all $\mathfrak{q},\tilde{\mathfrak{q}}\in \mathbf{Q}$, we have, with $u_{\mathfrak{q}}=u_{\mathfrak{a},\mathfrak{q}}$ and $u_{\tilde{\mathfrak{q}}}=u_{\mathfrak{a},\tilde{\mathfrak{q}}}$,
\[
\| \mathfrak{q}-\tilde{\mathfrak{q}}\|_{L^2(\Omega)}\le C\|\mathfrak{q}u_{\mathfrak{q}}^2-\tilde{\mathfrak{q}}u_{\tilde{\mathfrak{q}}}^2\|_{L^2(\Omega)},
\]
where $C=C(n,\Omega, \mu, \underline{\mathfrak{q}}, \overline{\mathfrak{q}},m)$.
\end{theorem}

Theorem \ref{main-theorem1} and \ref{main-theorem3} are to our knowledge new. Theorem \ref{main-theorem1} is in some sense optimal but cannot be extended in a neighborhood of an arbitrary $\mathfrak{q}$. The case $q\le 0$ and $j=2$ has already be considered. Different H\"older stability inequalities were established by Choulli and Triki \cite{CT1,CT2}.  Alessandrini \cite{Al} proved a H\"older stability inequality  in any interior subdomain for a set of unknown  coefficients $\mathfrak{q}$ defined by an implicit relation (Alessandrini \cite[Inequality (1.5)]{Al}), from the knowledge of the energy density $\mathfrak{q}u_{\mathfrak{q}}^2$.

\subsection{Conditional H\"older  stability}

Pick two constants $0<\mathfrak{q}_-\le \mathfrak{q}_+<\lambda_1(\Omega )$ and set 
\[
\mathcal{Q}=\{\mathfrak{q}\in L^\infty (\Omega);\; \mathfrak{q}_-\le -\mathfrak{q}\le \mathfrak{q}_+\}.
\]

Fix $\varrho >0$ sufficiently large in such a way that
\[
\mathcal{Q}_\varrho=\{ \mathfrak{q}\in \mathcal{Q}\cap C^{0,1}(\overline{\Omega});\; \|\mathfrak{q}\|_{C^{0,1}(\overline{\Omega})}\le \varrho\}\ne \emptyset .
\]

Let $\mathfrak{q}\in \mathcal{Q}$ and $r\ge 2$. Noting that $(I_n,\mathfrak{q})\in \mathcal{D}$, we deduce from Corollary \ref{corollary1.1} (Subsection \ref{subsection2.1}) that the BVP \eqref{e1}, in which we take $\mathfrak{a}=I_n$, has unique solution $u_{\mathfrak{q}}\in W^{2,r}(\Omega)$.

\begin{theorem}\label{theorem-hs1}
Suppose that $\Omega$ is of class $C^3$, $r>n$, $\mathfrak{a}=I_n$ and $f\in H^{5/2}(\Gamma)\cap W^{2-1/r,r}(\Gamma)$ is so that $f>0$ on $\Gamma$. Then,  for any $\mathfrak{q},\tilde{\mathfrak{q}}\in \mathcal{Q}_\varrho$, we have
\begin{equation}\label{hs2}
\|\mathfrak{q}-\tilde{\mathfrak{q}}\|_{L^\infty (\Omega)}\le C\varrho^{1-\gamma}\|u_{\mathfrak{q}}-u_{\tilde{\mathfrak{q}}}\|_{L^2(\Omega)}^{\gamma /3},
\end{equation}
where $C=C(n,\Omega,\mathfrak{q}_-,\mathfrak{q}_+,f,r)>0$ and $0<\gamma=\gamma (n,\Omega,\mathfrak{q}_-,\mathfrak{q}_+,f)<1$ are constants.
\end{theorem}

We consider more general $\mathfrak{a}$ in Theorem \ref{theorem-hs3} below.

\begin{theorem}\label{theorem-hs3}
Assume that $\Omega$ is of class $C^3$ and $f\in H^{5/2}(\Gamma)$ is non constant. Let $\omega \Subset \Omega$, $(\lambda ,\mu)\in\mathcal{J}$ and $\varrho >0$. Then, for any $(\mathfrak{a},\mathfrak{q}), (\tilde{\mathfrak{a}},\mathfrak{q})\in \mathcal{D}_{\lambda ,\mu}^\bullet$ so that $\mathfrak{q}\in C^{0,1}(\overline{\Omega})$, $\| \mathfrak{a}\|_{C^{0,1}(\overline{\Omega})}\le \varrho$, $\| \tilde{\mathfrak{a}}\|_{C^{0,1}(\overline{\Omega})}\le \varrho$ and $ \mathfrak{a}=\tilde{\mathfrak{a}}$ on $\Gamma$, we have, with $u_{\mathfrak{a}}=u_{\mathfrak{a},\mathfrak{q}}$ and $u_{\tilde{\mathfrak{a}}}=u_{\tilde{\mathfrak{a}}, \mathfrak{q}}$,
\[
\|\mathfrak{a}-\tilde{\mathfrak{a}}\|_{C(\overline{\omega})}\le C\varrho^{1-\gamma}\|u_{\mathfrak{a}}-u_{\tilde{\mathfrak{a}}}\|_{L^2(\Omega)}^{\gamma/3},
\]
where  $C=C(n,\Omega, \omega ,\lambda ,\mu, \mathfrak{q} ,f)$ and $0<\gamma =\gamma(n,\Omega, \omega ,\lambda ,\mu ,f)$ are constants.
\end{theorem} 

The preceding two theorems can be used to obtain stability inequalities for the quantitative dynamic elastography (at least for the simplified model described above).
Theorem \ref{theorem-hs1} can be interpreted as conditional stability estimate of the problem of recovering  the density from the knowledge of  tissue displacement, assuming that the shear modulus is known and it is identically equal to $1$. While Theorem \ref{theorem-hs3} gives an interior H\"older stability estimate of recovering the shear modulus when the density is supposed to be known. Here again the internal data consists in the tissue displacement.

It is worth pointing out that local H\"older stability inequalities were established by Bonito, Cohen, DeVore, Petrova and Welper \cite{BCDPW} for the problem of determining the conductivity in the equation
\[
-\mathrm{div}(\mathfrak{a}\nabla u)=f\quad \mbox{in}\; \Omega,\quad u_{|\Gamma}=0,
\]
from the knowledge of the corresponding solution, $u=u_{\mathfrak{a}}$. The approach carried out in  \cite{BCDPW} in quite different from the one we used in the present paper.

Fix $0<\beta <1$ and, for  $0<\varkappa \le \Lambda$, define then $\mathscr{D}_{\varkappa ,\Lambda}$ as the set of couples $(\mathfrak{a},\mathfrak{q})$ satisfying $\mathfrak{a}\in C^{1,\beta}(\overline{\Omega})$, $\mathfrak{q}\in C^{0,\beta}(\overline{\Omega})$ and
\[
\mathfrak{q}\ge 0,\quad \mathfrak{a}\ge \varkappa \quad \mbox{and}\quad \|\mathfrak{a}\|_{C^{1,\beta}(\overline{\Omega})}+\|\mathfrak{q}\|_{C^{0,\beta}(\overline{\Omega})}\le \Lambda .
\]

Suppose that $f\in C^{2,\beta}(\Gamma)$. Then, in light of Gilbarg and Trudinger \cite[Theorem 6.6, page 98 and Theorem 6.14, page 107]{GT}, the BVP \eqref{e1} admits a unique solution $u_{\mathfrak{a},\mathfrak{q}}(f)\in C^{2,\beta}(\overline{\Omega})$ so that
\begin{equation}\label{he}
\|u_{\mathfrak{a},\mathfrak{q}}(f)\|_{C^{2,\beta}(\overline{\Omega})}\le K,\quad \mbox{for all}\; (\mathfrak{a},\mathfrak{q})\in \mathscr{D}_{\varkappa ,\Lambda} ,
\end{equation}
where $K=K(n,\Omega ,\beta ,\varkappa ,\Lambda,f)$ is a constant.

\begin{theorem}\label{theorem-hs2}
Let $f_1,f_2\in C^{2,\beta}(\Gamma)$ with $f_1>0$.  Assume that $\Omega$ is of class $C^{2,\beta}$, $h=f_2/f_1$ is non constant and  the set of critical points of $h$ consists  of its extrema. For all $(\mathfrak{a},\mathfrak{q}),(\tilde{\mathfrak{a}},\tilde{\mathfrak{q}})\in \mathscr{D}_{\varkappa ,\Lambda}$ satisfying $(\mathfrak{a},\mathfrak{q})=(\tilde{\mathfrak{a}},\tilde{\mathfrak{q}})$ on $\Gamma$, we have, with $u_j=u_{\mathfrak{a},\mathfrak{q}}(f_j)$ and $\tilde{u}_j=u_{\tilde{\mathfrak{a}},\tilde{\mathfrak{q}}}(f_j)$, $j=1,2$,
\[
\|\mathfrak{a}-\tilde{\mathfrak{a}}\|_{C^{1,\beta}(\overline{\Omega})}+\|\mathfrak{q}-\tilde{\mathfrak{q}}\|_{C^{0,\beta}(\overline{\Omega})}\le \left(\|u_1-\tilde{u}_1\|_{C(\overline{\Omega})}+\|u_2-\tilde{u}_2\|_{C(\overline{\Omega})}\right)^\gamma ,
\]
where  $C=C(n,\Omega  ,\beta ,\varkappa ,\Lambda ,f_1,f_2)>0$ and $0<\gamma=\gamma (n,\Omega  ,\beta ,\varkappa ,\Lambda ,f_1,f_2)<1$ are constants.
\end{theorem}

In the case of dimension two functions called almost two-to-one, as it is defined in Nachman, Tamasan and Timonov \cite{NTT}, have no other critical points  than their extrema. A larger class of functions admitting such property consists in quantitatively unimodal functions. (e.g. Alessandrini and Nesi \cite{AN} for a precise definition). 

Theorem \ref{theorem-hs2} establishes conditional H\"older stability estimate of the quantitative photoacoustic tomopography consisting in determining simultaneously  the diffusion and the absorption coefficients from two internal energy densities, corresponding to two well-chosen illuminations. This theorem was already established in Bonnetier, Choulli and Triki \cite{BCT} when the two illuminations are generated from two point sources located outside the medium. 

Theorem \ref{theorem-hs2} improve the one  obtained in Alessandrini, Di Cristo, Francini and Vessella \cite{ADFV} under the assumptions that $h$ is quantitatively unimodal and $\Omega$ is diffeomorphic to the unit ball. We are convinced  that the hypotheses of Theorem \ref{theorem-hs2} on the boundary data are the best possible guaranteeing H\"older's type stability inequality in the whole domain.

The key point in the proof of Theorem \ref{theorem-hs2} is a uniform lower bound for the gradient that we obtain by a compactness argument (see Proposition \ref{proposition-glb}).

The rest of this text is devoted to the proof of the results stated in this introduction. The local Lipschitz stability inequalities are proved in Section \ref{section2}. While the proofs of the conditional H\"older stability inequalities are given in Section \ref{section3}.

The technical results we used, without a proof, in Section \ref{section3} were already established in Bonnetier, Choulli and Triki \cite{BCT}, Choulli and Triki \cite{CT1} and Choulli and Triki \cite{CT2}. These references contain a detailed proof of these technical results.

For simplicity reason we only considered the case when the bilinear form
\[
\mathfrak{b}(u,v)=\int_\Omega \mathfrak{a}\nabla u\cdot\nabla vdx+\int_\Omega \mathfrak{q}uvdx
\]
is coercive on $H_0^1(\Omega )$. However the non homogenous BVP \eqref{e1} is solvable provided that $0$ is not an eigenvalue of the operator $-\mathrm{div}(\mathfrak{a}\nabla \cdot )+\mathfrak{q}$, under Dirichlet boundary condition, and this property remains true if $(\mathfrak{a},\mathfrak{q})$ is substituted by $(\mathfrak{\tilde{a}},\tilde{\mathfrak{q}})$ sufficiently ``close''  to $(\mathfrak{a},\mathfrak{q})$. We believe that the results obtained in the present work may be extended, at least partially, to coefficients $(\mathfrak{a},\mathfrak{q})$ so that $0$ is not an eigenvalue of the operator $-\mathrm{div}(\mathfrak{a}\nabla \cdot )+\mathfrak{q}$, under Dirichlet boundary condition.

\section{Local Lipschitz stability inequalities}\label{section2}

\subsection{Solvability of the BVP \eqref{e1} in $W^{2,r}(\Omega)$} \label{subsection2.1}

It is contained in the following theorem. 

\begin{theorem}\label{theorem1.1}
Let  $r\ge 2$.  
\\
$\mathrm (i)$ For any $(\mathfrak{a},\mathfrak{q})\in \mathcal{D}$, the linear mapping
\[
\mathcal{P}_{\mathfrak{a},\mathfrak{q}}:u\in W^{2,r}(\Omega)\mapsto (-\mathrm{div}(\mathfrak{a}\nabla u)+\mathfrak{q}u,u_{|\Gamma})\in L^r(\Omega )\times W^{2-1/r,r}(\Gamma)
\]
is an isomorphism.
\\
$\mathrm (ii)$ Let  $(\mathfrak{a}_0,\mathfrak{q}_0)\in \mathcal{D}$. Then there exists  a constant $\delta=\delta (\mathfrak{a}_0,\mathfrak{q}_0)>0$ so that, for any $\mathfrak{q}\in  B_{L^\infty(\Omega)}(0,\delta)$, 
\[
\mathcal{P}_{\mathfrak{a}_0,\mathfrak{q}_0+\mathfrak{q}}:u\in W^{2,r}(\Omega)\mapsto (-\mathrm{div}(a\nabla u)+\mathfrak{q}u,u_{|\Gamma})\in L^r(\Omega )\times W^{2-1/r,r}(\Gamma)
\]
is an isomorphism with
\[
\|\mathcal{P}_{\mathfrak{a}_0,\mathfrak{q}_0+\mathfrak{q}}^{-1}\|_{\mathrm{op}}\le 2\|\mathcal{P}_{\mathfrak{a}_0,\mathfrak{q}_0}^{-1}\|_{\mathrm{op}},\quad \mbox{for all}\; \mathfrak{q}\in  B_{L^\infty(\Omega)}(0,\delta).
\]
\end{theorem}

\begin{proof}
(i) Follows by modifying slightly the proof of Choulli, Hu and Yamamoto \cite[Theorem 4.2]{CHY}.
\\
(ii) Pick $(F,f)\in L^r(\Omega )\times W^{2-1/r,r}(\Gamma)$ and define the mapping $T$ as follows
\[
T:W^{2,r}(\Omega )\rightarrow W^{2,r}(\Omega ):u\mapsto \mathcal{P}_{\mathfrak{a}_0,\mathfrak{q}_0}^{-1}(-\mathfrak{q}u+F,f).
\]
Clearly, we have
\[
T(u_1)-T(u_2)=\mathcal{P}_{\mathfrak{a}_0,\mathfrak{q}_0}^{-1}(-\mathfrak{q}(u_1-u_2),0).
\]
Hence
\[
\|T(u_1)-T(u_2)\|_{W^{2,r}(\Omega)}\le  \|\mathcal{P}_{\mathfrak{a}_0,\mathfrak{q}_0}^{-1}\|_{\mathrm{op}}\|\mathfrak{q}\|_{L^\infty(\Omega)}\|u_1-u_2\|_{W^{2,r}(\Omega)}.
\]

Let $\delta =1/\left[2 \|\mathcal{P}_{\mathfrak{a}_0,\mathfrak{q}_0}^{-1}\|_{\mathrm{op}}\right]$. Whence, if $\mathfrak{q}$ satisfies $\|\mathfrak{q}\|_{L^\infty(\Omega)}<\delta$ then
\begin{equation}\label{1.1}
\|T(u_1)-T(u_2)\|_{W^{2,r}(\Omega)}\le (1/2)\|u_1-u_2\|_{W^{2,r}(\Omega)}.
\end{equation}
According to Banach's fixed point theorem, $T$ admits a unique fixed point $u^\ast\in W^{2,r}(\Omega )$. In other words, we proved that there exists a unique $u^\ast \in W^{2,r}(\Omega )$ satisfying $\mathcal{P}_{\mathfrak{a}_0,\mathfrak{q}_0+\mathfrak{q}}u^\ast =(F,f)$. Furthermore, we get in light of \eqref{1.1}
\begin{align*}
\|\mathcal{P}_{\mathfrak{a}_0,\mathfrak{q}_0+\mathfrak{q}}^{-1}(F,f)&\|_{W^{2,r}(\Omega )}=\|u^\ast\|_{W^{2,r}(\Omega )}
\\
&\le \|T(u^\ast)-T(0)\|_{W^{2,r}(\Omega )}+\|T(0)\|_{ W^{2,r}(\Omega )}
\\
&\le (1/2)\|u^\ast\|_{W^{2,r}(\Omega)}+ \|\mathcal{P}_{\mathfrak{a}_0,\mathfrak{q}_0}^{-1}\|_{\mathrm{op}}\|(F,f)\|_{L^r(\Omega )\times W^{2-1/r,r}(\Gamma)}
\\
&\le (1/2)\|\mathcal{P}_{\mathfrak{a}_0,\mathfrak{q}_0+\mathfrak{q}}^{-1}(F,f)\|_{W^{2,r}(\Omega )}
\\
&\hskip 2cm + \|\mathcal{P}_{\mathfrak{a}_0,\mathfrak{q}_0}^{-1}\|_{\mathrm{op}}\|(F,f)\|_{L^r(\Omega )\times W^{2-1/r,r}(\Gamma)}
\end{align*}
and then 
\[
\|\mathcal{P}_{\mathfrak{a}_0,\mathfrak{q}_0+\mathfrak{q}}^{-1}\|_{\mathrm{op}}\le 2\|\mathcal{P}_{\mathfrak{a}_0,\mathfrak{q}_0}^{-1}\|_{\mathrm{op}}.
\]
The proof is then complete.
\end{proof}

\begin{corollary}\label{corollary1.1}
For any $(\mathfrak{a},\mathfrak{q})\in \mathcal{D}$ and $f\in W^{2-1/r,r}(\Gamma)$, the BVP \eqref{e1} has a unique solution $u=\mathcal{P}_{\mathfrak{a},\mathfrak{q}}^{-1}(0,f)$.
\end{corollary}

\subsection{Differentiability properties}\label{subsection2.2}

Fix $(\mathfrak{a}_0,\mathfrak{q}_0)\in \mathcal{D}$ and $f\in W^{2-1/r,r}(\Gamma )$ non identically equal to zero. Let $\delta =\delta(\mathfrak{a}_0,\mathfrak{q}_0)$ be as in Theorem  \ref{theorem1.1} (ii).  For notational convenience we use in the sequel the notations 
\[
\mathscr{S}_{\mathfrak{q}}=\mathcal{P}_{\mathfrak{a}_0,\mathfrak{q}_0+\mathfrak{q}}^{-1},\quad \mbox{for each}\; \mathfrak{q}\in B_{L^\infty (\Omega)}(0,\delta )
\]
and $\varpi =2\|\mathcal{P}_{\mathfrak{a}_0,\mathfrak{q}_0}^{-1}\|_{\mathrm{op}}$. That is, we have, according to Theorem  \ref{theorem1.1} (ii),
\begin{equation}\label{1.2}
\|\mathscr{S}_{\mathfrak{q}}\|_{\mathrm{op}}\le\varpi ,\quad \mbox{for each}\; \mathfrak{q}\in B_{L^\infty (\Omega)}(0,\delta ).
\end{equation}

Define 
\[
\Psi : B_{L^\infty (\Omega)}(0,\delta)\rightarrow W^{2,r}(\Omega ):\mathfrak{q} \mapsto \mathscr{S}_{\mathfrak{q}}(0,f).
\]

We claim that the mapping $\Psi$ is Lipschitz continuous. Indeed, for $\mathfrak{q}_1,\mathfrak{q}_2\in B_{L^\infty (\Omega)}(0,\delta)$, we have
\[
\mathscr{S}_{\mathfrak{q}_1}(0,f)-\mathscr{S}_{\mathfrak{q}_2}(0,f)=\mathscr{S}_{\mathfrak{q}_1}(F,0),
\]
with
\[
F=(\mathfrak{q}_2-\mathfrak{q}_1)\mathscr{S}_{\mathfrak{q}_2}(0,f).
\]
We find by applying twice inequality \eqref{1.2} 
\[
\| \mathscr{S}_{\mathfrak{q}_1}(0,f)-\mathscr{S}_{\mathfrak{q}_2}(0,f)\|_{W^{2,r}(\Omega)}
\le  \varpi ^2\|f\|_{W^{2-1/r,r}(\Gamma )}\|\mathfrak{q}_1-\mathfrak{q}_2\|_{L^\infty (\Omega)}.
\]
That is, we have
\begin{equation}\label{1.3}
\| \Psi(\mathfrak{q}_1)-\Psi(\mathfrak{q}_2)\|_{W^{2,r}(\Omega)}\le \mathfrak{c}\|\mathfrak{q}_1-\mathfrak{q}_2\|_{L^\infty (\Omega)},
\end{equation}
where $\mathfrak{c}=\varpi ^2\|f\|_{W^{2-1/r,r}(\Gamma )}$.

Let $\mathfrak{q}\in B_{L^\infty (\Omega)}(0,\delta)$ and consider the linear map
\[
L_{\mathfrak{q}}:\mathfrak{p}\in L^\infty (\Omega)\mapsto \mathscr{S}_{\mathfrak{q}}(-\mathfrak{p}\Psi (\mathfrak{q}),0)\in W^{2,r}(\Omega ).
\]
In light of \eqref{1.2} we have
\[
\|L_{\mathfrak{q}}(\mathfrak{p})\|_{W^{2,r}(\Omega )}\le \mathfrak{c}\|\mathfrak{p}\|_{L^\infty (\Omega)},
\]
implying that $L_{\mathfrak{q}}$ is bounded.

Next, let $\mathfrak{p}\in L^\infty (\Omega)$ so that  $\mathfrak{p}+\mathfrak{q}\in B_{L^\infty (\Omega)}(0,\delta)$. Then it is not difficult to check that
\[
\Psi (\mathfrak{q}+\mathfrak{p})-\Psi (\mathfrak{q})-L_{\mathfrak{q}}(\mathfrak{p})=
\mathscr{S}_{\mathfrak{q}}(-\mathfrak{p} [\Psi (\mathfrak{q}+\mathfrak{p})- \Psi (\mathfrak{q})],0).
\]
We get by applying inequality \eqref{1.2} and then inequality \eqref{1.3}
\[
\|\Psi (\mathfrak{q}+\mathfrak{p})-\Psi (\mathfrak{q})-L_{\mathfrak{q}}(\mathfrak{p})\|_{W^{2,r}(\Omega )}\le \mathfrak{c} \varkappa \|\mathfrak{p}\|_{L^\infty (\Omega)}^2.
\]
This shows that $\Psi$ is Fr\'echet differentiable at $\mathfrak{q}$. The differential of $\Psi$ at $\mathfrak{q}$, denoted by $\Psi '(\mathfrak{q})$, is then given by
\begin{equation}\label{1.4}
\Psi '(\mathfrak{q})(\mathfrak{p})=\mathscr{S}_{\mathfrak{q}}(-\mathfrak{p}\Psi (\mathfrak{q}),0),\quad \mbox{for all}\; \mathfrak{p}\in L^\infty (\Omega).
\end{equation}

Let us now prove that 
\[
\Psi': B_{L^\infty (\Omega)}(0,\delta)\rightarrow \mathscr{B}(L^\infty (\Omega), W^{2,r}(\Omega))
\]
is continuous. To this end, let $\mathfrak{q}\in B_{L^\infty (\Omega)}(0,\delta)$ and $\hat{\mathfrak{q}}\in L^\infty (\Omega)$  so that $\mathfrak{q}+\hat{\mathfrak{q}}\in B_{L^\infty (\Omega)}(0,\delta)$. We get in light of formula \eqref{1.4}, where $\mathfrak{p}\in L^\infty (\Omega)$,
\begin{align*}
\Psi '(\mathfrak{q}+\hat{\mathfrak{q}})(\mathfrak{p})-\Psi '(\mathfrak{q})(\mathfrak{p})
&=\mathscr{S}_{\mathfrak{q}+\hat{\mathfrak{q}}}(-\mathfrak{p}\Psi (\mathfrak{q}+\hat{\mathfrak{q}}),0)-\mathscr{S}_{\mathfrak{q}}(-\mathfrak{p} \Psi (\mathfrak{q}),0)
\\
&=\mathscr{S}_{\mathfrak{q}}(-\mathfrak{p}[\Psi (\mathfrak{q}+\hat{\mathfrak{q}})-\Psi (\mathfrak{q})],0)+\mathscr{S}_{\mathfrak{q}}(-\hat{\mathfrak{q}}\Psi (\mathfrak{q}+\hat{\mathfrak{q}}),0).
\end{align*}

We can proceed as before to derive, with the aid of inequalities \eqref{1.2} and \eqref{1.3}, the following estimate
\[
\|\Psi' (\mathfrak{q}+\hat{\mathfrak{q}})-\Psi' (\mathfrak{q})\|_{\mathrm{op}}\le C_0\|\hat{\mathfrak{q}}\|_{L^\infty (\Omega)},
\]
where $C_0>0$ is a constant independent of $\hat{\mathfrak{q}}$. This shows that $\Psi'$ is continuous at $\mathfrak{q}$. In other words, we proved that $\Psi$ in continuously Fr\'echet differentiable in $B_{L^\infty (\Omega)}(0,\delta)$.

\subsection{Proof of Theorem \ref{main-theorem1}}

We give the proof for $j=1$. That for $j=2$ is quite similar. 

 According to $C^{2,\theta}$-H\"older regularity, we get $\mathscr{S}_0(0,f)\in C^{2,\theta}(\overline{\Omega})$. Furthermore, in light of the strong maximum principle (e.g. Gilbarg and Trudinger \cite[Theorem 3.5, page 35]{GT}), we have $\mathscr{S}_0(0,f)>\min_\Gamma f$ in $\Omega$. 

Let $\Psi$ be as in Subsection \ref{subsection2.2} with $\mathfrak{q}_0=0$ and introduce the mapping
\[
\Phi :B_{L^\infty (\Omega)}(0,\delta) \rightarrow L^\infty (\Omega): \mathfrak{q}\mapsto \Phi (\mathfrak{q})= \mathfrak{q}\Psi (\mathfrak{q}).
\]
Since $\Psi$ is continuously Fr\'echet differentiable then so is $\Phi$. We have in addition
\[
\Phi '(\mathfrak{q})(\mathfrak{p})=\mathfrak{p}\Psi(\mathfrak{q})+\mathfrak{q}\Psi'(\mathfrak{q})(\mathfrak{p}),\quad \mbox{for all}\; \mathfrak{p}\in L^\infty (\Omega).
\]
In particular
\[
\Phi'(0)(\mathfrak{p})=\mathfrak{p}\Psi(0)\; (=\mathfrak{p}\mathscr{S}_0(0,f)),\quad \mbox{for all}\; \mathfrak{p}\in L^\infty (\Omega).
\]
We define the linear map $\ell:L^\infty (\Omega)\rightarrow L^\infty (\Omega)$ by
\[
\ell(\mathfrak{h})=[\Psi(0)]^{-1}\mathfrak{h},\quad \mathfrak{h}\in L^\infty (\Omega).
\]

Clearly, $\ell$ is bounded with
\[
\|\ell\|_{\mathrm{op}}\le \|[\Psi(0)]^{-1}\|_{L^\infty (\Omega)}.
\]
We can check that $\ell$ is the inverse of $\Phi'(0)$. Therefore, according to the inverse function theorem $\Phi$ is a diffeomophism from a neighborhood $\mathcal{U}$ of $0$ in $L^\infty (\Omega)$ onto a neighborhood $\mathcal{V}$ of $0$ in $L^\infty (\Omega)$. Whence the expected inequality follows.

\subsection{Proof of Theorem \ref{main-theorem3}}

The proof is inspired by that of Bal, Ren, Uhlmann and Zhou \cite[Theorem 3.1]{BRUZ}.

Pick $\mathfrak{q},\tilde{\mathfrak{q}}\in \mathbf{Q}$, and set $u=u_{\mathfrak{q}}$, $\tilde{u}=u_{\tilde{\mathfrak{q}}}$, $v=\mathfrak{q}u^2$ and $\tilde{v}=\tilde{\mathfrak{q}}\tilde{v}^2$.  In light of Gilbarg and Trudinger \cite[Theorem 8.1, page 179]{GT} we have
\begin{equation}\label{mt3-1}
u\ge m, \quad \tilde{u}\ge m\quad\quad \mbox{on}\; \overline{\Omega}.
\end{equation}

On the other hand straightforward computations give
\[
\mbox{div}(\mathfrak{a}\nabla (u-\tilde{u}))+\sqrt{\mathfrak{q}\tilde{\mathfrak{q}}}(u-\tilde{u})=(\sqrt{\mathfrak{q}}+\sqrt{\tilde{\mathfrak{q}}})(\sqrt{v}-\sqrt{\tilde{v}}).
\]
Taking into account that $u-\tilde{u}\in H_0^1(\Omega)$, we find by applying Green's formula
\begin{align}
\int_\Omega \mathfrak{a}\nabla (u-\tilde{u})\cdot \nabla (u-\tilde{u})dx&-\int_\Omega \sqrt{\mathfrak{q}\tilde{\mathfrak{q}}}(u-\tilde{u})^2dx\label{mt3-2}
\\
&=\int_\Omega (\sqrt{\mathfrak{q}}+\sqrt{\tilde{\mathfrak{q}}})(\sqrt{\tilde{v}}-\sqrt{v})(u-\tilde{u})dx.\nonumber
\end{align}
But
\begin{equation}\label{mt3-3}
\int_\Omega \sqrt{\mathfrak{q}\tilde{\mathfrak{q}}}(u-\tilde{u})^2dx\le \overline{\mathfrak{q}}\int_\Omega (u-\tilde{u})^2dx,
\end{equation}
and, using that $\mathfrak{a}\in \mathcal{A}_\mu$ and Poincar\'e's inequality, we find
\begin{equation}\label{mt3-4}
\int_\Omega \mathfrak{a}\nabla (u-\tilde{u})\cdot \nabla (u-\tilde{u})dx\ge \mu ^{-1}\lambda_1(\Omega)\int_\Omega (u-\tilde{u})^2dx.
\end{equation}
Therefore \eqref{mt3-3} and \eqref{mt3-4} in \eqref{mt3-2} yield
\[
(\mu ^{-1}\lambda_1(\Omega)-\overline{\mathfrak{q}})\int_\Omega (u-\tilde{u})^2dx\le \int_\Omega (\sqrt{\mathfrak{q}}+\sqrt{\tilde{\mathfrak{q}}})(\sqrt{\tilde{v}}-\sqrt{v})(u-\tilde{u})dx,
\]
which, combined with Cauchy-Schwarz's inequality, entails
\begin{equation}\label{mt3-5}
\| u-\tilde{u}\|_{L^2(\Omega )}\le \frac{2\sqrt{\overline{\mathfrak{q}}}\mu}{\lambda_1(\Omega)-\mu\overline{\mathfrak{q}}}\| \sqrt{\tilde{v}}-\sqrt{v}\|_{L^2(\Omega )}.
\end{equation}

Also, elementary calculations enable us to establish the following identity 
\[
\sqrt{\tilde{\mathfrak{q}}}-\sqrt{\mathfrak{q}}=\frac{\sqrt{\mathfrak{q}\tilde{\mathfrak{q}}}}{\sqrt{v}}(u-\tilde{u})+\frac{\sqrt{\mathfrak{q}}}{\sqrt{v}}(\sqrt{\tilde{v}}-\sqrt{v}).
\]
Whence
\begin{equation}\label{mt3-6}
\|\sqrt{\tilde{\mathfrak{q}}}-\sqrt{\mathfrak{q}}\|_{L^2(\Omega)}\le \frac{\overline{q}}{\sqrt{\underline{\mathfrak{q}}}m}\| u-\tilde{u}\|_{L^2(\Omega )}+\frac{\sqrt{\overline{q}}}{\sqrt{\underline{\mathfrak{q}}}m}\| \sqrt{\tilde{v}}-\sqrt{v}\|_{L^2(\Omega )}.
\end{equation}

Let
\[
C=\frac{2\overline{q}\sqrt{\overline{\mathfrak{q}}}\mu}{\sqrt{\underline{\mathfrak{q}}}m(\lambda_1(\Omega)-\mu\overline{\mathfrak{q}})}+\frac{\sqrt{\overline{q}}}{\sqrt{\underline{\mathfrak{q}}}m}.
\]
Then \eqref{mt3-5} together with \eqref{mt3-6} imply
\[
\|\sqrt{\tilde{\mathfrak{q}}}-\sqrt{\mathfrak{q}}\|_{L^2(\Omega)}\le C\| \sqrt{\tilde{v}}-\sqrt{v}\|_{L^2(\Omega )}.
\]
To complete the proof it is sufficient to use the following inequalities
\begin{align*}
&\|\tilde{\mathfrak{q}}-\mathfrak{q}\|_{L^2(\Omega)}=\|(\sqrt{\tilde{\mathfrak{q}}}+\sqrt{\mathfrak{q}})(\sqrt{\tilde{\mathfrak{q}}}-\sqrt{\mathfrak{q}})\|_{L^2(\Omega)}\le 2\sqrt{\overline{\mathfrak{q}}}\|\sqrt{\tilde{\mathfrak{q}}}-\sqrt{\mathfrak{q}}\|_{L^2(\Omega)},
\\
& \| \sqrt{\tilde{v}}-\sqrt{v}\|_{L^2(\Omega )}=\left\| \frac{\tilde{v}-v}{\sqrt{\tilde{v}}+\sqrt{v}}\right\|_{L^2(\Omega )}\le \frac{1}{2\sqrt{\underline{\mathfrak{q}}}m}\| \tilde{v}-v\|_{L^2(\Omega )}.
\end{align*}

\begin{remark}
{\rm
We observe that the uniqueness result corresponding to Theorem \ref{main-theorem3} holds for a larger class of unknown coefficients $\mathfrak{q}$, only under the weaker assumption that $f$ is non identically equal to zero. We refer to Choulli and Triki \cite[Theorem 2.1]{CT1} for further details.
}
\end{remark}

\section{H\"older Stability inequalities}\label{section3}

\subsection{Proof of Theorem \ref{theorem-hs1}}\label{subsection3.1}

Let $0<\gamma \le 1$. We say $\mathcal{W}\subset L_+^1(\Omega )=\{w\in L^\infty (\Omega);\; w\ge 0\}$ is a uniform set of weights for the interpolation inequality 
\begin{equation}\label{int1}
\|\phi\|_{L^\infty (\Omega)}\le C\|\phi\|_{C^{0,\gamma}(\overline{\Omega})}^{1-\mu}\|\phi w\|_{L^1(\Omega)}^\mu,
\end{equation}
if the constants $C>0$ and $0<\mu <1$ can be chosen independently of $w\in \mathcal{W}$ and $\phi\in C^{0,\gamma}(\overline{\Omega})$.

We note that our choice of $r$ guarantees that $W^{2,r}(\Omega )$ is continuously embedded in $C(\overline{\Omega})$. Let then $\tilde{\mathfrak{e}}$ denotes the norm of this embedding. 

Fix $0<m\le \tilde{\mathfrak{e}}M$ and set
\begin{align*}
\mathscr{S}=\{u\in W^{2,r}(\Omega );\; -\Delta u+\mathfrak{q}u=0\; \mbox{in}\; \Omega &,\; \mbox{for some}\; \mathfrak{q}\in \mathcal{Q},
\\
&\|u\|_{W^{2,r}(\Omega)}\le M\; \mbox{and}\; |u_{|\Gamma}|\ge m\}.
\end{align*}

Minors modifications in the proof of Choulli and Triki \cite[Theorem 2.2]{CT2} yield the following result.

\begin{theorem}\label{theorem-wii}
$\mathscr{W}=\{w=u^2;\; u\in \mathscr{S}\}$ is a uniform set of weights for the weighted interpolation inequality \eqref{int1}, with $C=C(n,\Omega,r,\mathfrak{q}_-,\mathfrak{q}_+,m,M)$ and $\mu=\mu (n,\Omega,r,\mathfrak{q}_-,\mathfrak{q}_+,m,M)$.
\end{theorem}

We know that, according to Choulli, Hu and Yamamoto \cite[Theorem 4.2]{CHY},  there exists a constant $M=M(n,\Omega,r ,f,\mathfrak{q}_-,\mathfrak{q}_+)$ so that
\begin{equation}\label{ae1}
\|u_{\mathfrak{q}}\|_{W^{2,r}(\Omega )}\le M,\quad \mbox{for all}\; \mathfrak{q}\in \mathcal{Q}.
\end{equation}

In light of Theorem \ref{theorem-wii} we have the following consequence.

\begin{corollary}\label{corollary-wii}
Let $W=\{w=u_{\mathfrak{q}}^2;\; \mathfrak{q}\in \mathcal{Q}\}$. Suppose that $f>0$ on $\Gamma$. Then $W$ is a uniform set of weights for the weighted interpolation inequality \eqref{int1}, with $C=C(n,\Omega,r,\mathfrak{q}_-,\mathfrak{q}_+,f)$ and $\mu=\mu (n,\Omega,r,\mathfrak{q}_-,\mathfrak{q}_+,f)$.
\end{corollary}

We are now ready to complete the proof of Theorem \ref{theorem-hs1}. We pick $\mathfrak{q},\tilde{\mathfrak{q}}\in \mathcal{Q}_\varrho$ and, for sake of  simplicity, we set $u=u_{\mathfrak{q}}$ and $\tilde{u}=u_{\tilde{\mathfrak{q}}}$.

We have
\[
\left\{ \begin{array}{ll} -\Delta u=-qu\in H^1(\Omega)\quad \mbox{in}\; \Omega, \\ u_{|\Gamma}=f. \end{array}\right.
\]
In light of Lions and Magenes \cite[Theorem 5.4 in page 165]{LM} we obtain, since $f\in H^{5/2}(\Gamma)$, that $u\in H^3(\Omega)$ and, for some constant $c=c(n,\Omega)$, we have
\[
\|u\|_{H^3(\Omega)}\le c\left(\|qu\|_{H^1(\Omega)}+\|f\|_{H^{5/2}(\Gamma )}\right).
\]
This and \eqref{ae1} yield in a straightforward manner that
\begin{equation}\label{0.3.1}
\|u\|_{H^3(\Omega)}\le C,
\end{equation}
where $C=C(n,\Omega ,f,r,\mathfrak{q}_-,\mathfrak{q}_+)$.

Similarly, $\tilde{u} \in H^3(\Omega)$ and 
\begin{equation}\label{0.3.2}
\|\tilde{u}\|_{H^3(\Omega)}\le C,
\end{equation}
with $C$ as in \eqref{0.3.1}.

Using the identity
\[
(\mathfrak{q}-\tilde{\mathfrak{q}})u=\Delta (u-\tilde{u})+\tilde{\mathfrak{q}}(\tilde{u}-u),
\]
we find
\[
\|(\mathfrak{q}-\tilde{\mathfrak{q}})u\|_{L^2(\Omega)}\le \|u-\tilde{u}\|_{H^2(\Omega)}+\mathfrak{q}_+\|\tilde{u}-u\|_{L^2(\Omega )}.
\]
This inequality together with the interpolation inequality, where $c=c(n,\Omega)$ is a constant,
\begin{equation}\label{ii}
\|w\|_{H^2(\Omega)}\le c\|w\|_{H^3(\Omega)}^{2/3}\|w\|_{L^2(\Omega)}^{1/3},\quad w\in H^3(\Omega),
\end{equation}
\eqref{0.3.1} and \eqref{0.3.2} imply
\begin{equation}\label{0.3.3}
\|(\mathfrak{q}-\tilde{\mathfrak{q}})u\|_{L^2(\Omega)}\le C\|\tilde{u}-u\|_{L^2(\Omega )}^{1/3}
\end{equation}
Here and until the end of this subsection $C=C(n,\Omega,r,\mathfrak{q}_-,\mathfrak{q}_+,f)>0$ is a generic constant.
From Corollary \ref{corollary-wii}  there exists a constant $0<\mu=\mu (n,\Omega,\mathfrak{q}_-,\mathfrak{q}_+,f)<1$ such that
\[
\|\mathfrak{q}-\tilde{\mathfrak{q}}\|_{L^\infty (\Omega )}\le C\|\mathfrak{q}-\tilde{\mathfrak{q}}\|_{C^{0,1}(\Omega )}^{1-\mu}\|(\mathfrak{q}-\tilde{\mathfrak{q}})u^2\|_{L^1(\Omega)}^\mu.
\]
Hence
\[
\|\mathfrak{q}-\tilde{\mathfrak{q}}\|_{L^\infty (\Omega )}\le C\varrho^{1-\mu}\|(\mathfrak{q}-\tilde{\mathfrak{q}})u\|_{L^2(\Omega)}^\mu \|u\|_{L^2(\Omega)}^\mu.
\]

Using \eqref{0.3.1} in order to get
\[
\|\mathfrak{q}-\tilde{\mathfrak{q}}\|_{L^\infty(\Omega )}\le C\varrho^{1-\mu}\|(\mathfrak{q}-\tilde{\mathfrak{q}})u\|_{L^2(\Omega)}^\mu .
\]
This estimate together with \eqref{0.3.3} give
\[
\|\mathfrak{q}-\tilde{\mathfrak{q}}\|_{L^\infty(\Omega )}\le C\varrho^{1-\mu}\|u-\tilde{u}\|_{L^2(\Omega)}^{\mu /3}.
\]
This is the expected inequality.

\subsection{Proof of Theorem \ref{theorem-hs3}}

In this proof $\mbox{sgn}_0$ denotes the sign function given by: $\mbox{sgn}_0(t)=-1$ if $t<0$, $\mbox{sgn}_0(0)=0$ and $\mbox{sgn}_0(t)=1$ if $t>0$. Let $(\mathfrak{a},\mathfrak{q}), (\tilde{\mathfrak{a}},\mathfrak{q})\in \mathcal{D}_{\lambda,\mu }^\bullet$ with $\mathfrak{a}=\tilde{\mathfrak{a}}$ on $\Gamma$ and
\begin{equation}\label{ths3-1}
\| \mathfrak{a}\|_{C^{0,1}(\overline{\Omega})}\le \varrho,\quad \| \tilde{\mathfrak{a}}\|_{C^{0,1}(\overline{\Omega})}\le \varrho.
\end{equation}

For notational convenience we set $u=u_{\mathfrak{a}}$ and $\tilde{u}=u_{\tilde{\mathfrak{a}}}$. We obtain after some straightforward computations
\[
\mbox{div}(|\mathfrak{a}-\tilde{\mathfrak{a}}|\nabla u)=\mbox{sgn}_0(\mathfrak{a}-\tilde{\mathfrak{a}})[\mathfrak{q}(u-\tilde{u})+\mbox{div}(\tilde{\mathfrak{a}}\nabla (u-\tilde{u}))].
\]
Whence
\[
\int_\Omega \mbox{div}(|\mathfrak{a}-\tilde{\mathfrak{a}}|\nabla u)udx =\int_\Omega\mbox{sgn}_0(\mathfrak{a}-\tilde{\mathfrak{a}})[\mathfrak{q}(u-\tilde{u})+\mbox{div}(\tilde{\mathfrak{a}}\nabla (u-\tilde{u}))]udx.
\]
Taking into account that $\mathfrak{a}=\tilde{\mathfrak{a}}$ on $\Gamma$, we get by applying Green's formula to the left hand side of the last identity
\begin{equation}\label{ths3-2}
\int_\Omega |\mathfrak{a}-\tilde{\mathfrak{a}}||\nabla u|^2dx=\int_\Omega\mbox{sgn}_0(\mathfrak{a}-\tilde{\mathfrak{a}})[\mathfrak{q}(u-\tilde{u})+\mbox{div}(\tilde{\mathfrak{a}}\nabla (u-\tilde{u}))]udx.
\end{equation}

Reducing \eqref{e1} to a BVP with homogeneous boundary condition, we prove in a straightforward manner that
\begin{equation}\label{ths3-3}
\|u\|_{H^1(\Omega )}\le C,
\end{equation}
where $C=C(n,\Omega ,\mu ,\lambda ,\varrho, f)$ is a  constant.  We then proceed as in Theorem \ref{theorem-hs1} in order to get
\begin{equation}\label{ths3-3.1}
\|u\|_{H^3(\Omega )}\le C.
\end{equation}
Here and henceforward $C=C(n,\Omega ,\mu ,\lambda ,\varrho,\mathfrak{q}, f)$ is a generic constant.

 We use \eqref{ths3-2} and \eqref{ths3-3} in order to get
\begin{equation}\label{ths3-4}
\| |\mathfrak{a}-\tilde{\mathfrak{a}}||\nabla u|^2\|_{L^1(\Omega )}\le C\|u-\tilde{u}\|_{H^2(\Omega)}.
\end{equation}

As \eqref{ths3-3} and \eqref{ths3-3.1}  remain valid when $u$ is substituted by $\tilde{u}$, we obtain by using \eqref{ii} together with \eqref{ths3-3.1}
\[
\|u-\tilde{u}\|_{H^2(\Omega)}\le C\|u-\tilde{u}\|_{L^2(\Omega)}^{1/3}.
\]
This inequality in \eqref{ths3-4} entails
\begin{equation}\label{ths3-5}
\| |\mathfrak{a}-\tilde{\mathfrak{a}}||\nabla u|^2\|_{L^1(\Omega )}\le C\|u-\tilde{u}\|_{L^2(\Omega)}^{1/3}.
\end{equation}

On the other hand, we can proceed similarly to the proof of Bonnetier, Choulli and Triki \cite[Lemma 3.7]{BCT}, by applying Garofalo and Lin \cite[Theorem 2.1]{GL} instead of Bonnetier, Choulli and Triki \cite[Lemma 3.6]{BCT}, in order to obtain
\begin{equation}\label{ths3-6}
\| \mathfrak{a}-\tilde{\mathfrak{a}}\|_{C(\overline{\omega})}\le C\| \mathfrak{a}-\tilde{\mathfrak{a}}\|_{C^{0,1}(\overline{\Omega} )}^{1-\gamma}\| |\mathfrak{a}-\tilde{\mathfrak{a}}||\nabla u|^2\|_{L^1(\Omega )}^\gamma ,
\end{equation}
where $0<\gamma =\gamma (n,\Omega ,\omega, \mu ,\lambda ,\varrho, f)<1$ and $C=C(n,\Omega ,\omega,\mu ,\lambda ,\varrho, f)$ are constants.

We end up getting the expected inequality by putting together \eqref{ths3-5} and \eqref{ths3-6}.

\subsection{Uniform lower bound for the gradient at the boundary}\label{subsection3.2}

Fix $0<\gamma < \nu <1$ and, for $0<\sigma_0\le \sigma_1$, set
\[
\Sigma=\left\{ \sigma \in C^{1,\nu}(\overline{\Omega});\; \sigma \ge \sigma_0\; \mbox{and}\; \|\sigma\|_{C^{1,\nu}(\overline{\Omega})}\le \sigma_1\right\}.
\]
Pick $\varphi \in C^{2,\gamma}(\Gamma)$ be non constant so that its critical points are its extrema. Consider then the BVP 
\begin{equation}\label{e2}
\left\{
\begin{array}{l}
\mbox{div}(\sigma\nabla u)=0\quad \mbox{in}\; \Omega,
\\
u_{|\Gamma}=\varphi.
\end{array}
\right.
\end{equation}

As $C^{1,\nu}(\overline{\Omega})$ is continuously embedded in $C^{1,\gamma}(\overline{\Omega})$, with reference to Gilbarg and Trudinger \cite[Theorem 6.6, page 98 and Theorem 6.14, page 107]{GT} we deduce that, for any $\sigma \in \Sigma$, the BVP \eqref{e2} has a unique solution $u_\sigma \in C^{2,\gamma}(\overline{\Omega})$ so that
\begin{equation}\label{glb1}
\|u_\sigma \|_{C^{2,\gamma}(\overline{\Omega})}\le C,
\end{equation}
where $C=C(n,\Omega ,\varphi,\sigma_0,\sigma_1,\gamma,\nu )$ is a constant.

\begin{proposition}\label{proposition-glb}
There exists a constant $\eta =\eta (n,\Omega ,\varphi,\sigma_0,\sigma_1,\gamma,\nu )>0$ so that
\begin{equation}\label{glb2}
|\nabla u_\sigma (x)|\ge \eta \quad \mbox{for all}\; (x,\sigma )\in \Gamma \times \Sigma .
\end{equation}
\end{proposition}

\begin{proof}
Let $\sigma \in \Sigma$. We first note that, according to the strong maximum principle, $u_\sigma$ achieves both its maximum and its minimum on $\Gamma$. That is, the maximum and and the minimum of $u_\sigma$ coincide with those of $\varphi$. But according to Hopf's lemma (e.g. Gilbarg and Trudinger \cite[Lemma 3.4, page 34]{GT}), if $x\in \Gamma$ is an extremum point then $|\partial_\nu u(x)|>0$. On the other hand, according to the assumption on $\varphi$, we have $|\nabla_\tau u_\sigma  (x)|=|\nabla_\tau\varphi (x)|>0$ if $x$ is not an extremum point of $\varphi$, where $\nabla_\tau$ stands for the tangential gradient. In consequence $|\nabla u_\sigma (x)|>0$ for any $x\in \Gamma$.

Let $\Sigma_+=\{\sigma \in C^{1,\gamma}(\overline{\Omega}); \sigma >0\}$ and consider  the mapping
\[
T:(x,\sigma)\in \Gamma \times \Sigma_+\rightarrow [0,\infty ): (x,\sigma )\mapsto |\nabla u_\sigma(x)| .
\]
Fix $\sigma \in \Sigma_+$. Let $\sigma '\in C^{1,\gamma}(\overline{\Omega})$ so that $\|\sigma '\|_{C^{1,\gamma}(\overline{\Omega})}\le 1$ and $\sigma +\sigma '>\min \sigma/2$. We have 
\[
\left\{
\begin{array}{l}
\mbox{div}(\sigma\nabla (u_\sigma-u_{\sigma+\sigma'}))=\mbox{div}(\sigma'\nabla u_{\sigma+\sigma'})\quad \mbox{in}\; \Omega,
\\
(u_\sigma-u_{\sigma+\sigma'})_{|\Gamma}=0,
\end{array}
\right.
\]
We can apply twice Gilbarg and Trudinger \cite[Theorem 6.6, page 98]{GT} in order  to get
\begin{equation}\label{glb3}
\|u_\sigma-u_{\sigma+\sigma'}\|_{C^{2,\gamma}(\overline{\Omega})}\le C\|\sigma '\|_{C^{1,\gamma}(\overline{\Omega})},
\end{equation}
where $C=C(n,\Omega ,\varphi,\sigma_0,\sigma_1,\gamma,\nu )$ is a constant.

Therefore, for any $x,x'\in \Gamma$, we have
\begin{align*}
||\nabla u_\sigma (x)|-|\nabla u_{\sigma+\sigma'}(x')||&\le |\nabla u_\sigma (x)|-\nabla u_{\sigma}(x')|+|\nabla u_\sigma (x')|-\nabla u_{\sigma+\sigma'}(x')|
\\
&\le C\left(|x-x'|+ \|\sigma '\|_{C^{1,\gamma}(\overline{\Omega})}\right).
\end{align*}
That is,  the  mapping $T$ is continuous. We complete the proof by noting that, according to Gilbarg and Trudinger \cite[Lemma 6.36, page 136]{GT}, $\Gamma \times \Sigma$ is a compact subset of $\Gamma \times \Sigma_+$.
\end{proof}

\subsection{Proof of Theorem \ref{theorem-hs2}}\label{subsection3.3}

Hereafter, for $\delta>0$, we use the notations
\begin{align*}
&\Omega_\delta=\{z\in \Omega ;\; \mbox{dist}(z,\Gamma)\le \delta\},
\\
&\Omega^\delta=\{z\in \Omega ;\; \mbox{dist}(z,\Gamma)\ge \delta\}.
\end{align*}

\begin{lemma}\label{lemma-pos}
Under the condition $\min_\Gamma f >0$, we have
\[
u_{\mathfrak{a},\mathfrak{q}}(f) \ge \varepsilon \quad \mbox{for any}\; (\mathfrak{a},\mathfrak{q})\in \mathscr{D}_{\varkappa ,\Lambda},
\]
where $\varepsilon =\varepsilon (n,\Omega ,\varkappa ,\Lambda , f)>0$ is a constant
\end{lemma}

\begin{proof}
Let $(\mathfrak{a},\mathfrak{q})\in\mathscr{D}_{\varkappa ,\Lambda}$, $u=u_{\mathfrak{a},\mathfrak{q}}(f)$, and let $K$ be the constant in \eqref{he}. If $x\in \Omega$ and $y\in \Gamma$ are so that $|x-y|\le \delta$, then
\[
u(x)\ge u(y)-K|x-y|^\beta\ge m-K\delta^\beta ,
\]
where $m=\min_\Gamma f$.

We get by taking $\delta = [m/(2K)]^{1/\beta}$
\begin{equation}\label{pos1}
u(x)\ge m/2\quad \mbox{for any}\; x\in \Omega_\delta .
\end{equation}
We apply Harnak's inequality (e.g. Gilbarg and Trudinger \cite[Theorem 8.21, page 199]{GT}) in order to get
\[
\sup_{\Omega^{\delta/2}} u\le c\inf_{\Omega^{\delta/2}} u,
\]
where $c=c(n,\Omega  ,\varkappa ,\Lambda , f)>0$ is a constant. 

Inequality \eqref{pos1} then yields 
\begin{equation}\label{pos2}
m/(2c)\le u(x)\quad \mbox{for any}\; x\in \Omega^{\delta/2}.
\end{equation}

The expected inequality follows by putting together \eqref{pos1} and \eqref{pos2}.
\end{proof}

As a consequence of estimate \eqref{he} and Lemma \ref{lemma-pos} we obtain, after making straightforward calculations, the following result.

\begin{corollary}\label{corollary-quo}
Let $(\mathfrak{a},\mathfrak{q})\in \mathscr{D}_{\varkappa ,\Lambda}$, $f_1,f_2\in C^{2,\beta}(\Gamma )$ with $f_1>0$. Set 
\[
w=w_{\mathfrak{a},\mathfrak{q}}=\frac{u_{\mathfrak{a},\mathfrak{q}}(f_2)}{u_{\mathfrak{a},\mathfrak{q}}(f_1)}, \quad h=\frac{f_2}{f_1}\quad  \mathrm{and}\quad \sigma =\mathfrak{a}u_{\mathfrak{a},\mathfrak{q}}^2(f_1).
\]
Then $w\in C^{2,\beta}(\overline{\Omega})$ is the solution of the BVP
\[
\mathrm{div} (\sigma \nabla w)=0\; \mathrm{in}\; \Omega ,\quad w_{|\Gamma}=h.
\]
Furthermore
\begin{equation}\label{quo}
\mu_0\le \sigma ,\quad \|\sigma\|_{C^{1,\beta}(\Gamma)}\le \mu_1 \quad \mathrm{and}\quad \|w\|_{C^{2,\beta}(\overline{\Omega})}\le M,
\end{equation}
for some positive constants $\mu_0=\mu_0 (n,\Omega  ,\varkappa ,\Lambda ,\beta, f_1)$, $\mu_1=\mu_1 (n,\Omega  ,\varkappa ,\Lambda , f_1)$ and $M=M(n,\Omega  ,\lambda ,\Lambda , \beta,f_1,f_2)$.
\end{corollary}

Let $w=w_{\mathfrak{a},\mathfrak{q}}$ be as in the preceding corollary. In light of Proposition \ref{proposition-glb} we have
\[
|\nabla w(y)|\ge \eta \quad \mbox{for any}\; y\in \Gamma ,
\]
for some constant $\eta=\eta (n,\Omega  ,\varkappa ,\Lambda ,\beta, f_1,f_2)>0$.

Pick $\delta >0$. Let $x\in \Omega_\delta$ and $y\in \Gamma$ so that $|x-y|\le \delta$. Then
\[
|\nabla w(x)|\ge |\nabla w(y)|-M\hat{\mathfrak{e}}\delta \ge \eta-M\hat{\mathfrak{e}}\delta ,
\]
where $\hat{\mathfrak{e}}$ is a constant depending only on  the embedding $C^{2,\beta}(\overline{\Omega})\hookrightarrow C^{1,1}(\overline{\Omega})$. We then fix $\delta >0$ sufficiently small in such a way that
\begin{equation}\label{equa1}
|\nabla w(x)|\ge  \eta/2 \quad \mbox{for any}\; x\in \Omega_\delta .
\end{equation}

We get  by applying Bonnetier, Choulli and Triki \cite[Corollary 3.1]{BCT} that
\begin{equation}\label{equa2}
C\rho^\upsilon\le \|\nabla w\|_{L^2(B(x,\rho))}^2\quad \mbox{for any}\; x\in \Omega ^\delta \; \mbox{and}\; 0<\rho<\delta ,
\end{equation}
where $C=C(n,\Omega  ,\varkappa ,\Lambda ,\beta, f_1,f_2,\delta )$ and $\upsilon=\upsilon(n,\Omega  ,\varkappa ,\Lambda ,\beta, f_1,f_2,\delta)$ are positive constants.

\begin{lemma}\label{lemma-wii}
If $w$ is as in Corollary \ref{corollary-quo} then 
\begin{equation}\label{equa3}
\|\phi\|_{C(\overline{\Omega})}\le C\|\phi\|_{C^{0,\beta}(\overline{\Omega})}^{1-\gamma}\|\phi|\nabla w|\|_{L^1(\Omega )}^\gamma , \quad \mbox{for any}\; \phi \in C^{0,\beta}(\overline{\Omega}),
\end{equation}
where $C=C(n,\Omega ,\beta ,\varkappa ,\Lambda ,f_1,f_2)>0$ and $0<\gamma=\gamma (n,\Omega ,\beta ,\varkappa ,\Lambda ,f_1,f_2)<1$ are constants.
\end{lemma}

\begin{proof}
By homogeneity it is sufficient to prove \eqref{equa3} with $\phi \in C^{0,\beta}(\overline{\Omega})$ satisfying $\|\phi\|_{C^{0,\beta}(\overline{\Omega})}=1$. 

For $x\in \Omega ^\delta$ and $y\in B(x,\rho)$, $0<\rho <\delta$, we have
\[
|f(x)|\le |f(y)|+\rho ^\beta.
\]
In consequence
\[
|\phi (x)|\int_{B(x,\rho)}|\nabla w(y)|^2dy\le \int_{B(x,\rho)}|\phi (y)||\nabla w(y)|^2dy+\rho ^\beta \int_{B(x,\rho)}|\nabla w(y)|^2dy .
\]
As $w$ is non constant, we have $\|\nabla w\|_{L^2(B(x,\rho)}\ne 0$ by the uniqueness of continuation property and hence
\[
|\phi (x)|\le \frac{\|\phi |\nabla w|^2\|_{L^1(\Omega)}}{\|\nabla w\|_{L^2(B(x,\rho))}^2}+\rho ^\beta .
\]
This and \eqref{equa2} yield
\[
C|\phi (x)|\le \rho^{-\upsilon}\|\phi |\nabla w|^2\|_{L^1(\Omega)}+\rho ^\beta \quad \mbox{for any}\; x\in \Omega^\delta \; \mbox{and}\; 0<\rho <\delta .
\]
That is, we have
\begin{equation}\label{equa4}
C\|\phi \|_{L^\infty (\Omega ^\delta )}\le \rho^{-\upsilon}\|\phi |\nabla w|^2\|_{L^1(\Omega)}+\rho ^\beta \quad \mbox{for any}\; 0<\rho <\delta .
\end{equation}

Next, assume that $\|\phi\|_{L^\infty (\Omega_\delta)}\ne 0$. Pick then $x_0\in \Omega_\delta$ so that $|\phi (x_0)|=\|\phi\|_{L^\infty (\Omega_\delta)}$. For $0<\rho <\delta $, we have
\[
|\phi (x_0)|\le |\phi (x)|+ \rho^\beta \quad x\in B(x_0,\rho)\cap \Omega_\delta .
\]  
Whence
\[
|\phi (x_0)|\le (\eta /2)^{-2}|\phi (x)||\nabla w|^2 +\rho^\beta \quad x\in B(x_0,\rho)\cap \Omega_\delta 
\]
implying
\[
|\phi (x_0)||B(x_0,\rho)\cap \Omega_\delta |\le (\eta /2)^{-2}\int_{B(x_0,\rho)\cap \Omega_\delta}|\phi (x)||\nabla w|^2 +\rho^\alpha |B(x_0,\rho)\cap \Omega_\delta|  .
\]
But since $\Omega$ has the uniform interior cone property, we have $|B(x_0,\rho)\cap \Omega_\delta |\ge c\rho ^n$, for any $0<\rho<\delta/2$, where $c=c(\Omega)$ is a constant. In consequence
\begin{equation}\label{equa5}
C\|\phi\|_{L^\infty (\Omega_\delta)}\le \rho ^{-n}\|\phi |\nabla w|^2\|_{L^1(\Omega)}+\rho^\beta \quad \mbox{for any}\; 0<\rho <\delta/2 .
\end{equation}

A combination of \eqref{equa4} and \eqref{equa5} gives
\begin{equation}\label{equa6}
C\|\phi\|_{L^\infty (\Omega )}\le \rho ^{-k}\|\phi |\nabla w|^2\|_{L^1(\Omega)}+\rho^\beta \quad \mbox{for any}\; 0<\rho <\delta/2 ,
\end{equation}
where $k=\max (n,\upsilon)$.

A well known argument consisting in minimizing the right hand side of \eqref{equa6} with respect to $\rho$ yields the expected inequality.
\end{proof}

Let $(\mathfrak{a},\mathfrak{q})$, $(\tilde{\mathfrak{a}},\tilde{\mathfrak{q}})\in \mathscr{D}_{\varkappa,\Lambda}$ satisfying $(\mathfrak{a},\mathfrak{q})=(\tilde{\mathfrak{a}},\tilde{\mathfrak{q}})$ on $\Gamma$. With the aid of the weighted interpolation inequality \eqref{equa3}, we can mimic the last part of the proof of Bonnetier, Choulli and Triki \cite[Theorem 1.1]{BCT} in order to prove the following stability inequality, with for $j=1,2$, $u_j=u_{\mathfrak{a},\mathfrak{q}}(f_j)$ and $\tilde{u}_j=u_{\tilde{\mathfrak{a}},\tilde{\mathfrak{q}}}(f_j)$,
\[
\|\mathfrak{a}-\tilde{\mathfrak{a}}\|_{C^{1,\beta}(\overline{\Omega})}+\|\mathfrak{q}-\tilde{\mathfrak{q}}\|_{C^{0,\beta}(\overline{\Omega})}\le \left(\|u_1-\tilde{u}_1\|_{C(\overline{\Omega})}+\|u_2-\tilde{u}_2\|_{C(\overline{\Omega})}\right)^\gamma ,
\]
where $C=C(n,\Omega ,\beta ,\varkappa ,\Lambda ,f_1,f_2)>0$ and $0<\gamma=\gamma (n,\Omega ,\beta  ,\varkappa ,\Lambda ,f_1,f_2)<1$ are constants. Theorem \ref{theorem-hs2} is then proved.

\subsection*{Acknowledgement} The author would like to thank the referee for his helpful comments which have contributed to the improvement of this work.


\begin{thebibliography}{99}

\bibitem{AC} G. S. Alberti and Y. Capdeboscq, 
\newblock Lectures on elliptic methods for hybrid inverse problems,
\newblock Soci\'et\'e Math\'ematique de France, Paris, 25, pp.vii + 230, 2018.

\bibitem{Al} G. Alessandrini, 
\newblock Global stability for a coupled physics inverse problem,
\newblock  Inverse Problems 30 (7) (2014), 075008, 10 pp.

\bibitem{AN} G. Alessandrini, V. Nesi, 
\newblock Quantitative estimates on Jacobians for hybrid inverse problems, 
\newblock Bulletin SUSU MMCS 8 (3)  (2015), 25-41.

\bibitem{ADFV} G. Alessandrini, M. Di Cristo, E. Francini and  S. Vessella,
\newblock Stability for quantitative photoacoustic tomography with well chosen illuminations, 
\newblock Ann. Mat. Pura e Appl. 196 (2) (2017), 395-406.


\bibitem{ACGTR} H. Ammari, Y. Capdeboscq, F. de Gournay, F. Triki and A. Rozanova-Pierrat,
\newblock Microwave imaging by elastic deformation,
\newblock  SIAM J. Appl. Math. 71 (6) (2011), 2112-2130.


\bibitem {BRUZ} G. Bal, K. Ren, G. Uhlmann and T. Zhou,
\newblock Quantitative thermo-acoustics and related problems, 
\newblock Inverse Prob.  27 (5), 055007, 2011. 

\bibitem{BS} G. Bal and J.C. Schotland,
\newblock  Inverse scattering and acousto-optic imaging, 
\newblock Phys. Rev. Letters, 104, 043902, 2010. 

\bibitem{BU} G. Bal and G. Uhlmann, 
\newblock Inverse diffusion theory of photoacoustics, 
\newblock Inverse Prob. 26 (2010):085010.

\bibitem{BCDPW} A. Bonito, A. Cohen, R. DeVore, G. Petrova and G.Welper, 
\newblock Diffusion coefficients estimation for elliptic partial differential equations,
\newblock  SIAM J. Math. Anal. 49 (2) (2017),1570-1592.

\bibitem{BCT} E. Bonnetier, M. Choulli and F. Triki, 
\newblock Stability for quantitative photoacoustic tomography  revisited,   
\newblock arXiv:1905.07914.


\bibitem{CHY} M. Choulli, G. Hu and M. Yamamoto, 
\newblock Stability inequality for a semilinear elliptic inverse problem,
\newblock Nonlinear Differ. Equ. Appl. 28, 37 (2021) 26 p.   

\bibitem{CT1} M. Choulli and F. Triki, 
\newblock New stability estimates for the inverse medium problem with internal data,
\newblock SIAM J. Math. Anal. 47 (3) (2015), 1778-1799.

\bibitem{CT2} M. Choulli and  F. Triki, 
\newblock H\"older stability for an inverse medium problem with internal data, 
\newblock Res. Math. Sci.  (2019), 6:9, 15 pp. 


\bibitem{GL} N. Garofalo and F.-H. Lin, 
\newblock Unique continuation for elliptic operators: a geometric-variational approach, 
\newblock Commun. Pure Appl. Math. 40 (3) (1987) 347-366.

\bibitem{GT} D. Gilbarg and N. S. Trudinger,
\newblock Elliptic partial differential equations of second order, Springer, Berlin, 1998.

\bibitem{LM} J.-L. Lions and E. Magenes, 
\newblock Non-homogeneous boundary value problems and applications, Vol. I. Translated from the French by P. Kenneth. Die Grundlehren der mathematischen Wissenschaften, Band 18, 
\newblock Springer-Verlag, New York-Heidelberg, 1972. xvi+357 pp.

\bibitem{NTT} A. Nachman, A. Tamasan,  A. Timonov, 
\newblock Conductivity imaging with a single measurement of boundary and interior data, 
Inverse Problems 23 (6) (2007), 2551-2563.



\end{thebibliography}
\end{document}